\documentclass[proceedings,submission%
]{dmtcs}

\usepackage[T1]{fontenc}
\usepackage[utf8]{inputenc}
\usepackage{subfigure}
\usepackage{amsmath, amssymb}
\usepackage{enumitem}

\newtheorem{lemma}{Lemma}[section]
\newtheorem{theorem}[lemma]{Theorem}

\newtheorem{proposition}[lemma]{Proposition}
\newtheorem{problem}[lemma]{Problem}
\newtheorem{conjecture}[lemma]{Conjecture}

\usepackage[round]{natbib}

\usepackage{tikz}
\usetikzlibrary{decorations,calc,through}
\usetikzlibrary{decorations.pathmorphing}
\usetikzlibrary{decorations.pathreplacing}

\newcommand{\T}{\mathcal{T}}
\newcommand{\G}{\mathcal{G}}
\newcommand{\R}{\mathbb{R}}

\newcommand{\N}{\mathbb{N}}
\newcommand{\Y}{\mathbb{Y}}

\newcommand{\F}{\mathcal{F}}

\newcommand{\Sym}[1]{\mathfrak{S}_{#1}}

\newcommand{\M}{\mathcal{M}}

\newcommand{\vol}{\operatorname{vol}}

\newcommand{\supp}{\operatorname{supp}}
\newcommand{\Area}{\operatorname{Area}}
\newcommand{\Alg}{\mathcal{P}}

\author{Maciej Do\l\k{e}ga\addressmark{1}
\and
Piotr \'Sniady\addressmark{1}\addressmark{2}}
\title[Polynomial functions on Young diagrams]{Polynomial functions on Young
diagrams arising from bipartite graphs}
\address{\addressmark{1}Instytut Matematyczny,
Uniwersytet Wrocławski,  \mbox{pl.\ Grunwaldzki~2/4,} 50-384
Wrocław, Poland\\
\addressmark{2}Instytut Matematyczny, Polska Akademia Nauk,
ul.~Śniadec\-kich 8, 00-956 Warszawa, Poland}
\begin{document}
\maketitle
\begin{abstract}
\paragraph{Abstract.}
We study the class of functions on the set of (generalized) Young diagrams
arising as the number of embeddings of bipartite graphs. We give a criterion for
checking when such a function is a polynomial function on Young diagrams (in the
sense of Kerov and Olshanski) in terms of combinatorial properties of the
corresponding bipartite graphs. Our method involves
development of a differential calculus of functions on the set of generalized
Young diagrams.

\paragraph{R\'esum\'e.}
Nous étudions la classe des fonctions sur l'ensemble des diagrammes de Young
(généralisés) qui sont définies comme des nombres d'injections de graphes
bipartites. Nous donnons un critère pour savoir si une telle fonction est une
fonctions polynomiale sur les diagrammes de Young (au sens de Kerov et
Olshanski) utilisant les propriétés combinatoires des graphes bipartites
correspondants. Notre méthode repose sur le développement d'un calcul
différentiel sur les fonctions sur les diagrammes de Young généralisés.
\end{abstract}

The full version of this extended abstract will be published elsewhere.

\section{Introduction}

\subsection{Prominent examples of polynomial functions on Young diagrams}
\label{subsec:examples}

The character $\chi^\lambda(\pi)$ of the symmetric group is usually considered
as a function of the permutation $\pi$, with the Young diagram $\lambda$ fixed.
Nevertheless, it was observed by \cite{KerovOlshanski1994} that for several
problems of the asymptotic representation theory it is convenient to do the
opposite: keep the permutation $\pi$ fixed and let the Young diagram $\lambda$
vary.  It should be stressed the Young diagram $\lambda$ is arbitrary, in
particular there are no restrictions on the number of boxes of
$\lambda$. In this way it is possible to study the structure of the series of
the symmetric groups $\Sym{1}\subset \Sym{2}\subset \cdots$ and their
representations in a uniform way. In order for this idea to be successful one
has to replace the usual characters $\chi^\lambda(\pi)$ by, so called,
\emph{normalized characters} $\Sigma_\pi(\lambda)$, namely for partitions
$\pi,\lambda$ such that $|\pi| = k$, $|\lambda| = n$ we define
$$ \Sigma_\pi(\lambda) =\begin{cases} \underbrace{n(n-1)\cdots (n-k+1)}_{k
\text{
factors}} \frac{\chi^\lambda(\pi,1^{n-k})}{\chi^\lambda(1^n)} &
\text{if } k\leq n, \\ 0 & \text{otherwise.} \end{cases}
$$

There are several other interesting examples of functions on the set of Young
diagrams which in some sense --- which will be specified in Section
\ref{subsec:algebra} --- are similar to the normalized characters; we recall
some of them in the following.

\emph{Free cumulants} $R_k(\lambda)$, introduced by \cite{Kerov2000talk}
and \cite{Biane2003}, are relatively simple functionals of the shape of the
Young diagram $\lambda$. Their advantage comes from the fact that the
normalized characters can be expressed in terms of free cumulants and that this
expression takes a particularly simple form \citep{Biane2003}.

There are other interesting functionals of the shape of the Young diagram and
\emph{fundamental functionals of shape}, which are defined in Section
\ref{sub:shape}, are very simple
examples. These functionals are quite useful and powerful in the context of
differential calculus on Young diagrams.

Jack symmetric functions \citep{Jack1970/1971} are a generalization of
Schur
functions and are indexed by an additional parameter $\alpha>0$. They can
be used to define \emph{Jack characters} $\Sigma^{(\alpha)}_\pi(\lambda)$ which
are a natural generalization of the normalized characters of the symmetric
groups. For some special values of $\alpha$ Jack symmetric functions become
well-known objects. For example, for $\alpha = 2$ we obtain so-called
\emph{zonal polynomial} which is a zonal spherical function for the Gelfand
pairs
$(\Sym{2n}, H_n)$, where $H_n$ denotes the
hyperoctahedral group. This example
has an important meaning in the
representation theory for Gelfand pairs. Jack characters are related
with $\alpha$-anisotropic Young diagrams which are deformed Young diagrams with
respect to the parameter $\alpha$. More precisely, for a Young diagram $\lambda
= (\lambda_1, \dots, \lambda_k)$ we can construct an $\alpha$-anisotropic Young
diagram $\alpha \lambda$ by rescaling the original diagram in one direction by
parameter $\alpha$, i.e. $\alpha \lambda = (\alpha \lambda_1,
\dots, \alpha \lambda_k)$. The connection between functions on
$\alpha$-anisotropic Young diagrams and Jack polynomials was obtained by
\cite{Kerov2000}.

\subsection{The algebra of polynomial functions on Young diagrams}
\label{subsec:algebra}

\cite{KerovOlshanski1994} defined  the \emph{algebra $\Alg$ of polynomial
functions on the set $\Y$ of Young diagrams}. This algebra is generated by every
example of a family of functions on Young diagrams which we presented in Section
\ref{subsec:examples}. The problem of how to express an element of one basis in
terms of the elements of another basis of this algebra is very fascinating and
it is related to Kerov polynomials, Goulden-Rattan polynomials and many other
combinatorial objects which are sometimes well-known, but sometimes far away
from being satisfactorily understood.

% This concept became very fruitful
% because this algebra turned out to have several quite distinct facets, each
% emphasizing another feature of the symmetric groups and their representations.
% We will review briefly some of these facets in the following.
% 
% Firstly, we can view the elements of the algebra $\Alg$ as functions on the
% set
% $\Y$ of (generalized) Young diagrams. In this approach there are several very
% natural algebraic bases of this algebra, each related to another aspect of the
% representation theory of the symmetric groups. We presented the most prominent
% examples in Section \ref{subsec:examples}.

The algebra $\Alg$ of polynomial functions on $\Y$ turns
out to be isomorphic to a subalgebra of the algebra of partial permutations
of \cite{IvanovKerov1999}. Therefore we can view the elements of
$\Alg$ as (linear combinations of) partial permutations. Since the
multiplication of functions on $\Y$
corresponds to convolution of central functions on partial
permutations, we see that the algebra $\Alg$ turns out
to be very closely related to the problems of computing connection coefficients
and multiplication of conjugacy classes in the symmetric groups. 
% It is
% remarkable that this subalgebra of the algebra of partial permutations which
% corresponds to $\Alg$ is isomorphic to a rather classic object: the algebra
% studied by \citeauthor{FarahatHigman} in \citeyear{FarahatHigman}.

The algebra $\Alg$ is
canonically isomorphic to the algebra of \emph{shifted symmetric functions}. 
% and it is convenient to identify these two algebras. 
The algebra of shifted symmetric functions is an important object in the
symmetric functions theory
and the isomorphism with the algebra $\Alg$ gives some new results in this field
due to \cite{OkounkovOlshanski1998}.

% Fourthly, [supersymmetric functions]

% The above collection of alternative ways of viewing the algebra $\Alg$ of
% polynomial functions on Young diagrams probably is not complete, but it
% already
% shows the richness of this structure. Several problems from the asymptotic
% representation theory of symmetric groups turned out to be equivalent to
% questions concerning the algebra $\Alg$ and relating various
% ways of viewing it and, in particular, finding relations between its various
% algebraic bases which is in particular a common question in lot of places in
% combinatorics.
% \textbf{Znacznie lepiej!}

\subsection{Numbers of colorings of bipartite graphs}
\label{sub:numbersofcolor}

The set of vertices of a bipartite graph $G$ will always be $V=V_1\sqcup V_2$
with the elements of $V_1$ (respectively, $V_2$) referred to as white
(respectively, black) vertices.

We consider a coloring $h$ of the white vertices in $V_1$ by columns of the
given Young diagram $\lambda$ and of the black vertices in $V_2$ by rows of the
given Young diagram $\lambda$. Formally, a coloring is a function $h : V_1\sqcup
V_2 \rightarrow \N$ and we say that this coloring is \emph{compatible} with a
Young diagram $\lambda$ if $(h(v_1), h(v_2)) \in \lambda$ (where $(h(v_1), h(v_2))$
denotes the box placed in $h(v_1)$th column and in $h(v_2)$th row) for each
edge $(v_1,
v_2)$ of $G$ with $v_1\in V_1$, $v_2\in V_2$. Alternatively, a coloring which
is compatible with $\lambda$ can be
viewed as a function which maps the edges of the bipartite graph to boxes of
$\lambda$ with a property that if edges $e_1,e_2$ share a common white
(respectively, black) vertex then $h(e_1)$ and $h(e_2)$ are in the same column
(respectively, the same row). We can think that such a coloring defines an
embedding of a graph $G$ into the Young diagram $\lambda$.
We denote by $N_G(\lambda)$ the number of colorings of $G$ which are
compatible with $\lambda$ which is the same as the number of embeddings of
$G$ into $\lambda$ by the above identification.

% The linear combinations of functions of this form will be
% called \emph{piecewise polynomial functions} on $\Y$. This algebra contains an
% algebra of polynomial functions on $\Y$ and the main question about this
% algebra is how to characterize piecewise polynomial functions on $\Y$ which
% are
% polynomial on $\Y$ in fact. This terminology rise up
% from the fact, that piecewise polynomial functions on $\Y$ generalize
% polynomial
% functions on $\Y$. \textbf{Coś mi tu nie gra, prewencyjnie zostawiam obie
% wersje.}
% 

% \begin{proposition}
% Let $G$ be a bipartite graph, and $G = \bigsqcup_{1 \leq i \leq n} G_i$ be a
% decomposition of $G$ on connected components. Then $N_G = \prod_{1 \leq i
% \leq n} N_{G_i}$.
% \end{proposition}
% 
% \begin{proof}
% It follows directly from the definition of $N_G$.
% \end{proof}

\subsection{Polynomial functions on Young diagrams and bipartite graphs}

Suppose that some interesting polynomial function $F\in\Alg$ is given. It turns
out that it is very convenient to write $F$ as a linear combination of the
numbers of embeddings $N_G$ for some suitably chosen bipartite graphs $G$:
\begin{equation}
\label{eq:expansion}
 F = \sum_{G} \alpha_G  N_G
\end{equation}
which is possible for any polynomial function $F$.
This idea was initiated by \cite{F'eray'Sniady-preprint2007} who found
explicitly such linear combinations for the normalized characters
$\Sigma_\pi(\lambda)$
% The original motivation of \cite{F'eray'Sniady-preprint2007} was to
% give a conceptual reformulation of a formula conjectured by
% \cite{Stanley-preprint2006} for the characters on multirectangular Young
% diagrams and which was proved by \cite{F'eray2010}.
and who used them to give new upper bounds on the characters of the symmetric
groups. Another application of this idea was given by
\cite*{DolegaF'eray'Sniady2008} who found explicitly the expansion of the
normalized character $\Sigma_{\pi}(\lambda)$ in terms of the free cumulants
$R_s(\lambda)$; such expansion is called Kerov polynomial
\citep{Kerov2000talk,Biane2003}. 

The above-mentioned two papers concern only the case when $F=\Sigma_{\pi}$ is
the normalized character, nevertheless it is not difficult to adapt them to
other cases for which the expansion of $F$ into $N_{G}$ is known. For example,
\cite{F'eray'Sniady2011} found also such a representation for the zonal
characters
and in this way found the Kerov polynomial for the zonal polynomials.

It would be very tempting to follow this path and to generalize these results
to other interesting polynomial functions on $\Y$. However, in order to do this
we need to overcome the following essential difficulty.
\begin{problem}
\label{prob:main}
For a given interesting polynomial function $F$ on the set of Young diagrams,
how to find explicitly the expansion \eqref{eq:expansion} of $F$ as a linear
combination of the numbers of colorings $N_G$? 
\end{problem}
This problem is too ambitious and too general to be tractable. In this article
we will tackle the following, more modest question.
\begin{problem}
\label{prob:modest}
Which linear combinations of the numbers of colorings $N_G$ are polynomial
functions on the set of Young diagrams? 
\end{problem}
Surprisingly, in some cases the answer to this more modest Problem
\ref{prob:modest} can be helpful in finding the answer to the more important
Problem \ref{prob:main}.

\subsection{How characterization of polynomial functions can be useful?}
\label{sub:characterization}

Jack shifted symmetric functions $J^{(\alpha)}_\mu$ with parameter $\alpha$
are indexed by Young diagrams and are characterized (up to a multiplicative
constant) by the following conditions:
\begin{enumerate}
[label=(\roman*)]
 \item 
\label{war-znikanie}
$J^{(\alpha)}_{\mu}(\mu) \neq 0$ and for each Young diagram $\lambda$ such
that $|\lambda| \leq |\mu|$ and $\lambda
\neq \mu$ we have $J^{(\alpha)}_{\mu}(\lambda) = 0$;
\item
\label{war-polynomial}
% $J^{(\alpha)}_{\mu}$ is an $\alpha$-shifted symmetric function or,
% equivalently, 
$J^{(\alpha)}_{\mu}$ is an $\alpha$-anisotropic polynomial function on the set
of Young diagrams, i.e.~the function $\lambda\mapsto
J^{(\alpha)}_{\mu}\left(\frac{1}{\alpha}\lambda\right)$ is a polynomial
function;
\item
\label{war-stopien}
$J^{(\alpha)}_{\mu}$ has degree equal to $|\mu|$ (regarded as a shifted
symmetric function).
\end{enumerate}
The structure on Jack polynomials remains mysterious and there are several open
problems concerning them. The most interesting for us are introduced and
investigated by \cite{Lassalle2008a,Lassalle2009}.

One possible way to overcome these difficulties is to write Jack shifted
symmetric functions in the form
$$ J_\mu^{(\alpha)}(\lambda) = \sum_{\pi\vdash |\mu|}
n^{(\alpha)}_{\pi}\ \Sigma^{(\alpha)}_{\pi}(\mu)\
\Sigma^{(\alpha)}_{\pi}(\lambda),$$
where $n^{(\alpha)}_\pi$ is some combinatorial factor which is out of scope of
the current paper and where $\Sigma^{(\alpha)}_{\pi}$, called Jack character,
is an $\alpha$-anisotropic polynomial function on the set of Young diagrams.
The problem is therefore reduced to finding the expansion 
\eqref{eq:expansion} for Jack characters (which is a special case of Problem
\ref{prob:main}). It is tempting to solve this problem by guessing the right
form of the expansion \eqref{eq:expansion} and then by proving that so defined
$J^{(\alpha)}_{\mu}$ have the required properties.

We expect that verifying a weaker version of condition \ref{war-znikanie},
namely:
\begin{enumerate}
[label=(i')]
 \item 
\label{war:i'weak}
For each Young diagram $\lambda$ such that
$|\lambda| <|\mu|$ we have
$J^{(\alpha)}_{\mu}(\lambda) = 0$
\end{enumerate}
should not be too difficult; sometimes it does not matter if in the definition
of $N_G(\lambda)$ we count all embeddings of the graph into the Young diagram
or we count only injective embeddings in which each edge of the graph is mapped
into a different box of $\lambda$. If this is the case then condition
\ref{war:i'weak} holds trivially if all graphs $G$ over which we sum have
exactly $|\mu|$ edges. Also condition \ref{war-stopien} would follow trivially.
The true difficulty is to check that condition \ref{war-polynomial} is fulfilled
which is exactly the special case of Problem \ref{prob:modest} (up to the small
rescaling related to the fact that we are interested now with
$\alpha$-anisotropic polynomial functions).

% \subsubsection{Application: New proofs of old results}
% 
% Special cases of Jack characters (which are generalizations of normalized
% characters of a symmetric group indexed by a positive real value $\alpha$
% \cite{Lassalle2009}, \cite{F'eray2010a}) are normalized
% character and zonal character, which correspond to $\alpha = 1$ and $\alpha =
% 2$ respectively. We will express them in functions
% $N_\G$. This was done by F\'eray and Śniady in
% \cite{F'eray'Sniady-preprint2007}
% and \cite{F'eray2010a} to prove some conjectures of Lassalle, but the main
% result of this paper gives us much shorter and simpler proof. In this point we
% can see, that having an expression of Jack characters in terms of $N_\G$
% provide
% us a quick answer for conjectures of Lassalle, hence functions $N_\G$
% seems to be very powerful and interesting tools.

\subsection{The main result}

The main result of this paper is Theorem \ref{theo:main} which gives a
solution to Problem \ref{prob:modest} by characterizing the linear combinations
of $N_G$ which are polynomial functions on $\Y$ in terms of a combinatorial
property of the underlying formal linear combinations of bipartite graphs $G$.

\subsection{Contents of this article}

In this article we shall highlight just the main ideas of the proof of Theorem
\ref{theo:main} because the whole proof is rather long and technical. In
particular we will briefly show the main conceptual ingredients: differential
calculus on $\Y$ and derivation of bipartite graphs.

Due to lack of space we were not able to show the full history of the presented
results and to give to
everybody the proper credits. For more history and bibliographical references we
refer to the full version
of this article \cite{Dolega'Sniady2010} which will be published elsewhere.

\section{Preliminaries}
\label{sec:preliminaries}

\subsection{Russian and French convention}

\begin{figure}[tbp]

    \begin{tikzpicture}

\begin{scope}[scale=0.5/sqrt(2),rotate=-45,draw=gray]

      \begin{scope}
          \clip[rotate=45] (-2,-2) rectangle (7.5,6.5);
          \draw[thin, dotted, draw=gray] (-10,0) grid (10,10);
          \begin{scope}[rotate=45,draw=black,scale=sqrt(2)]
%               \clip (0,0) rectangle (4.5,5.5);
              \draw[thin, dotted] (0,0) grid (15,15);
          \end{scope}
      \end{scope}

      \draw[->,thick] (-4.5,0) -- (4.5,0)
node[anchor=west,rotate=-45]{\textcolor{gray}{$z$}};
      \foreach \z in { -3, -2, -1, 1, 2, 3}
            { \draw (\z, -2pt) node[anchor=north,rotate=-45]
{\textcolor{gray}{\tiny{$\z$}}} -- (\z, 2pt); }

      \draw[->,thick] (0,-0.4) -- (0,9.5)
node[anchor=south,rotate=-45]{\textcolor{gray}{$t$}};
      \foreach \t in {1, 2, 3, 4, 5, 6, 7, 8, 9}
            { \draw (-2pt,\t) node[anchor=east,rotate=-45]
{\textcolor{gray}{\tiny{$\t$}}} -- (2pt,\t); }

      \begin{scope}[draw=black,rotate=45,scale=sqrt(2)]

          \draw[->,thick] (0,0) -- (6,0) node[anchor=west]{{{$x$}}};
          \foreach \x in {1, 2, 3, 4, 5}
              { \draw (\x, -2pt) node[anchor=north] {{\tiny{$\x$}}} -- (\x,
2pt); }

          \draw[->,thick] (0,0) -- (0,5) node[anchor=south] {{{$y$}}};
          \foreach \y in {1, 2, 3, 4}
              { \draw (-2pt,\y) node[anchor=east] {{\tiny{$\y$}}} -- (2pt,\y); }

          \draw[line width=3pt,draw=black] (5.5,0) -- (4,0) -- (4,1) -- (3,1) --
(3,2) -- (1,2) -- (1,3) -- (0,3) -- (0,4.5) ;
          \fill[fill=gray,opacity=0.1] (4,0) -- (4,1) -- (3,1) -- (3,2) -- (1,2)
-- (1,3) -- (0,3) -- (0,0) -- cycle ;

      \end{scope}
 
\end{scope}

\begin{scope}[xshift=10cm, yshift=-0.5cm, scale=0.5]
       \begin{scope}
          \clip (-4.5,0) rectangle (5.5,5.5);
          \draw[thin, dotted] (-6,0) grid (6,6);
          \begin{scope}[rotate=45,draw=gray,scale=sqrt(2)]
              \clip (0,0) rectangle (4.5,5.5);
              \draw[thin, dotted] (0,0) grid (6,6);
          \end{scope}
      \end{scope}

      \draw[->,thick] (-6,0) -- (6,0) node[anchor=west]{$z$};
      \foreach \z in {-5, -4, -3, -2, -1, 1, 2, 3, 4, 5}
            { \draw (\z, -2pt) node[anchor=north] {\tiny{$\z$}} -- (\z, 2pt); }

      \draw[->,thick] (0,-0.4) -- (0,6) node[anchor=south]{$t$};
      \foreach \t in {1, 2, 3, 4, 5}
            { \draw (-2pt,\t) node[anchor=east] {\tiny{$\t$}} -- (2pt,\t); }

% \begin{scope}
%           \clip (5.5,0) rectangle (15.5,5.5);
%           \draw[thin, dotted] (4,0) grid (16,6);
%           \begin{scope}[rotate=45,draw=gray,scale=sqrt(2)]
%               \clip (10,0) rectangle (4.5,5.5);
%               \draw[thin, dotted] (10,0) grid (16,6);
%           \end{scope}
%       \end{scope}
% 
%       \draw[->,thick] (4,0) -- (16,0) node[anchor=west]{$z$};
%       \foreach \z in {5, 6, 7, 8, 9, 11, 12, 13, 14, 15}
%             { \draw (\z, -2pt) node[anchor=north] {\tiny{$\z$}} -- (\z,
% 2pt); }
% 
%       \draw[->,thick] (10,-0.4) -- (10,6) node[anchor=south]{$t$};
%       \foreach \t in {1, 2, 3, 4, 5}
%             { \draw ($(10,0)+(-2pt,\t)$) node[anchor=east] {\tiny{$\t$}} --
% ($(10,0)+(2pt,\t)$); }

\begin{scope}[draw=gray,rotate=45,scale=sqrt(2)]

          \draw[->,thick] (0,0) -- (6,0) node[anchor=west,rotate=45]
{\textcolor{gray}{{$x$}}};
          \foreach \x in {1, 2, 3, 4, 5}
              { \draw (\x, -2pt) node[anchor=north,rotate=45]
{\textcolor{gray}{\tiny{$\x$}}} -- (\x, 2pt); }

          \draw[->,thick] (0,0) -- (0,5) node[anchor=south,rotate=45]
{\textcolor{gray}{{$y$}}};
          \foreach \y in {1, 2, 3, 4}
              { \draw (-2pt,\y) node[anchor=east,rotate=45]
{\textcolor{gray}{\tiny{$\y$}}} -- (2pt,\y); }

          \draw[line width=3pt,draw=black] (5.5,0) -- (4,0) -- (4,1) -- (3,1) --
(3,2) -- (1,2) -- (1,3) -- (0,3) -- (0,4.5) ;
          \fill[fill=gray,opacity=0.1] (4,0) -- (4,1) -- (3,1) -- (3,2) -- (1,2)
-- (1,3) -- (0,3) -- (0,0) -- cycle ;

      \end{scope}
\end{scope}

    \end{tikzpicture}

    \caption{Young diagram $(4,3,1)$ shown in the French and Russian
conventions.
             The solid line represents the profile of the Young diagram.
             The coordinates system $(z,t)$ corresponding to the
             Russian convention and the coordinate system $(x,y)$
             corresponding to the French convention are shown.}

    \label{fig:french}
\end{figure}
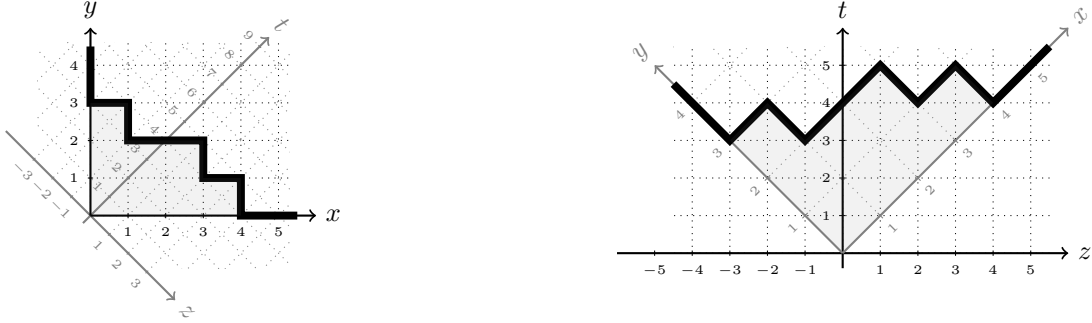

We will use two conventions for drawing Young diagrams: the \emph{French} one
in the $0xy$ coordinate system and the \emph{Russian} one in the $0zt$
coordinate system (presented on Figure
\ref{fig:french}).
Notice that the graphs in the Russian convention are created from the graphs in
the French convention by rotating counterclockwise by $\frac{\pi}{4}$ and by
scaling by a factor $\sqrt{2}$. Alternatively, this can be viewed as choice of
two
coordinate systems on the plane: $0xy$, corresponding to the French convention,
and $0zt$,
corresponding to the Russian convention. 
% The coordinates in both systems are
% related
% to each other by
% $$ \left\{
%   \begin{aligned}
%     z & = x-y, \\
%     t & = x+y,
%   \end{aligned} \right. \qquad \qquad  \left\{
%   \begin{aligned}
%     x & = \frac{z+t}{2}, \\
%     y & = \frac{t-z}{2}.
%   \end{aligned} \right. $$
For a point on the plane we will define its \emph{content} as its
$z$-coordinate.

In the French coordinates will use the plane $\R^2$ equipped with the standard
Lebesgue measure,
i.e.~the area of a unit square with vertices $(x,y)$ such that $x,y\in\{0,1\}$
is equal to $1$. This measure in the Russian coordinates
corresponds to a the Lebesgue measure on $\R^2$ multiplied by the factor $2$,
i.e.~
i.e.~the area of a unit square with vertices $(z,t)$ such that $z,t\in\{0,1\}$
is equal to $2$.

\subsection{Generalized Young diagrams}

We can identify a Young diagram drawn in the Russian convention with its
profile, see Figure \ref{fig:french}. It is therefore natural to define the set
of \emph{generalized Young diagrams $\Y$ (in the Russian convention)} as the set
of functions $\omega:\R\rightarrow\R_+$ which fulfill the following two
conditions:
\begin{itemize}
 \item $\omega$ is a Lipschitz function with constant $1$, i.e.\
$|\omega(z_1)-\omega(z_2)|\leq
|z_1-z_2|$,
\item $\omega(z)=|z|$ if $|z|$ is large enough.
\end{itemize}
We will define the support of $\omega$ in a natural way:
$$ \supp(\omega) = \overline{ \{z \in \R : \omega(z) \neq |z|\}} .$$

% At the first sight it might seem that we have defined the set $\Y$ of
% generalized Young diagrams in two different ways, but we prefer to think that
% these two definitions are just two conventions (French and Russian) for
% drawing
% the same object. This will not lead to confusions since it will be always
% clear
% from the context which of the two conventions is being used.

% [for simplicity, Young diagram= generalized Young diagram; we will not
% consider
% any conventional Young diagrams at all ]
% 
% 
% [maybe it makes sense to replace \emph{Young diagram in the Russian notation}
% by
% \emph{profile of a Young diagram}]

% \subsubsection{$\alpha$-anisotropic Young diagrams}
% 
% Following Kerov \cite{Kerov2000}, we define so called $alpha$-anisotropic
% Young
% diagrams, where for a given Young diagram $\lambda$, the $\alpha$-anisotropic
% Young diagram $\alpha \lambda$ is the Young diagram stretched anisotropically
% only along the $OX$ axis, i.e.
% $$\alpha \lambda = (\alpha \lambda_1,\alpha \lambda_2,\dots). $$

\subsection{Functionals of shape}
\label{sub:shape}

We define the fundamental functionals of shape $\lambda$ for integers $k\geq 2$
% \begin{equation}
% \label{eq:sk}
\[ S_k(\lambda) = (k-1) \iint_{(x,y)\in\lambda} (x-y)^{k-2} \ dx\
dy = \frac{1}{2} (k-1) \iint_{(z,t)\in\lambda} z^{k-2} \ dz \ dt, \]
% \end{equation}
where the first integral is written in the French and the second in the Russian
coordinates.
The family $(S_k)_{k\geq 2}$ generates the algebra $\Alg$ of polynomial
functions on
Young diagrams \citep*{DolegaF'eray'Sniady2008}.

\section{Differential calculus of functions on Young diagrams}
\label{sec:diff}

\subsection{Content-derivatives}

% [motivations-Gâteaux derivative]

Let $F$ be a function on the set of generalized Young diagrams and let $\lambda$
be a generalized Young diagram. We ask how quickly the value of $F(\lambda)$
would change if we change the shape of $\lambda$ by adding infinitesimal
boxes with content equal to $z$. In order to answer this informally
formulated question we define a derivative of $F$ with respect to content $z$;
this definition is inspired by the Gâteaux derivative. We say that
$$ \partial_{C_z} F (\lambda) = f(z) $$
if $f:\R\rightarrow\R$ is a continuous function such that for any $\epsilon>0$
and $C>0$ there exists $\delta>0$ such that
for any generalized Young diagrams $\omega_1$, $\omega_2$ supported on $[-C,C]$
such that
$\| \omega - \omega_i \|_{L^1}< \delta$ for $i\in\{1,2\}$
\begin{equation}
\label{eq:derivative}
 \left| F(\omega_1) - F(\omega_2) - \frac{1}{2} \int_\R f(z) \big(
\omega_1(z) - \omega_2(z) \big) \ dz \right| \leq \epsilon\ \| \omega_1
-\omega_2\|_{L^1}.
\end{equation}
The strange constant $\frac{1}{2}$  in the above definition appears because of
the fact that we are working with a Russian convention which rescales the length
and the height of the Young diagram by a factor $\sqrt{2}$, hence
$$ \Area(\lambda) = \frac{1}{2} \int_\R 
\big( \omega(z) - |z| \big) \ dz .$$

It can be shown using similar methods as in the case of a standard Gâteaux
derivative that a content-derivative has the following properties:
\begin{enumerate}[label=(\Alph*)]

\item
\label{lem:unique}
If the derivative $\partial_{C_z} F (\lambda)$ exists, then it is unique.

\item
\label{der:a}
The Leibniz rule holds, i.e. if $F_1,F_2$ are sufficiently smooth functions
then
$$ \partial_{C_z} F_1 F_2 = \left( \partial_{C_z} F_1 \right)  F_2+ F_1
\partial_{C_z} F_2. $$

\item
\label{der:b} 
For any integer $k\geq 2$
$$ \partial_{C_z} S_k = (k-1) z^{k-2}. $$

\end{enumerate}

The next proposition shows important properties of derivation of a
polynomial function on $\Y$.

\begin{proposition}
Let $F$ be a polynomial function on $\Y$.
\begin{itemize}
% [label=(\roman*)]
\item  For any Young diagram $\lambda$
the function $\R\ni z\mapsto \partial_{C_z} F (\lambda)$
is a polynomial.

\item For any $z_0\in\R$ the function
$ \Y\ni \lambda \mapsto \partial_{C_{z_0}} F (\lambda)$
is a polynomial function on $\Y$.

\item For any integer $k\geq 0$ the function
$ \Y\ni\lambda \mapsto [z^k] \partial_{C_z} F (\lambda)$
is a polynomial function on $\Y$.
\end{itemize}
\end{proposition}
\begin{proof}
By linearity it is enough to prove it for $F = \prod_{1 \leq i \leq n}
S_{k_i}$.
Then, thanks to the properties \ref{der:a} and \ref{der:b}, we have that
$$\partial_{C_z}F =
\sum_{1 \leq i \leq n} \prod_{\substack{1 \leq j \leq n \\ j \neq i}} S_{k_j}
(k_i-1) z^{k_i-2},$$
which is a polynomial in $z$ for fixed $\lambda$, and which is a polynomial
function on $\Y$ for fixed $z = z_0$. Moreover $[z^k]
\partial_{C_z} F (\lambda) $ is a linear combination of products of $S_{k_i}$,
hence it is a polynomial function on $\Y$, which finishes
the proof.
\end{proof}

The main result of this paper is that (in some sense) the opposite implication
is true as well
and thus it characterizes the polynomial functions on $\Y$.

In order to show it we would like to look at the content-derivative of
$N_G(\lambda)$, hence it is necessary to extend the domain of the function
$N_G$ to the set of generalized Young diagrams. This extension is very natural.
Indeed, we
consider a coloring $h : V_1\sqcup
V_2 \rightarrow \R_+$ and we say that this coloring is \emph{compatible} with
generalized Young diagram $\lambda$ if $(h(v_1), h(v_2)) \in \lambda$ for each
edge $(v_1,
v_2) \in V_1 \times V_2$ of $G$. If we fix the order of vertices in $V=V_1\sqcup
V_2$, we can think of a coloring $h$ as an element of
$\R_+^{|V|}$. Then we define
$$N_G(\lambda) = \vol\{h \in \R_+^{|V|} : h \text{ compatible with }
\lambda\}.$$ 
Notice that this is really an extension, i.e. this function restricted to the
set of ordinary Young diagrams is the same as $N_G$ which was defined in
Section \ref{sub:numbersofcolor}.

Just before we finish this section, let us state one more lemma which will be
helpful soon and which explains the connection between the usual derivation of a
function on the set of Young
diagrams when we change the shape of a Young diagram a bit, and the
content-derivative of this function.

\begin{lemma}
\label{lem:int}
Let $\R\ni t \mapsto \lambda_t $ be a sufficiently smooth trajectory in the set
of generalized Young diagrams
and let $F$ be a sufficiently smooth function on $\Y$.
Then
$$ \frac{d}{dt} F(\lambda_t) = \int_\R  \frac{1}{2} \frac{d\omega_t(z)}{dt} \
\partial_{C_z} F(\lambda_t)\  dz. $$
\end{lemma}
\begin{proof}
This is a simple consequence of equality \eqref{eq:derivative}.
% Strictly from \eqref{eq:derivative} we have the following:
% \begin{multline*}
% \left( \int_\R  \frac{1}{2} \frac{d\omega_t(z)}{dt} \
% \partial_{C_z} F(\lambda_t)\  dz \right) - \epsilon \left\|
% \frac{d\omega_t(z)}{dt} \right\|_{L^1} \leq \\
% \leq \frac{d}{dt} F(\lambda_t) \leq \left( \int_\R  \frac{1}{2}
% \frac{d\omega_t(z)}{dt} \
% \partial_{C_z} F(\lambda_t)\  dz \right) + \epsilon \left\|
% \frac{d\omega_t(z)}{dt} \right\|_{L^1}
% \end{multline*}
% for any $\epsilon > 0$. Taking $\epsilon \to 0$ we finish the proof.
\end{proof}

\section{Derivatives on bipartite graphs}
\label{sec:graphs}

Let $G$ be a bipartite graph. We denote
$$\partial_z G=\sum_e (G,e)$$
which is a formal sum (formal linear combination) which runs over all edges $e$
of $G$. We will think about the pair $(G,e)$ that it is graph $G$ with one edge
$e$ decorated with the symbol $z$. More generally, if $\G$ is a linear
combination of bipartite graphs, this definition extends by linearity.

% [in the future we might be considering more objects of this form, like
% $\partial_{z_i} G$ and then the edge $e$ will be decorated by $z_i$, but now
% we do not have to worry about it]

If $G$ is a bipartite graph with one edge decorated by the symbol $z$, we define
$$ \partial_x G = \sum_f G_{f\equiv z} $$
which is a formal sum which runs over all edges $f\neq z$
which share a common black vertex with the edge $z$. The symbol
$G_{f\equiv z}$ denotes the graph $G$ in which the edges $f$ and $z$
are glued together (which means that also the white vertices of $f$ and $z$
are glued together and that from the resulting graph all multiple edges are
replaced
by single edges). The edge resulting from gluing $f$ and $z$ will be decorated
by
$z$. More generally, if $\G$ is a linear combination of bipartite graphs,
this definition extends by linearity.

We also define
$$ \partial_y G = \sum_f G_{f\equiv z} $$
which is a formal sum which runs over all edges $f\neq z$
which share a common white vertex with the edge $z$.

\begin{conjecture}
\label{theo:rozn}
Let $\G$ be a linear combination of bipartite graphs with a property that
% \begin{equation} 
% \label{eq:1der}
\[
 \left( \partial_x + \partial_y \right) \partial_z \G = 0 . \]
% \end{equation}
Then for any integer $k\geq 1$
% \begin{equation}
% \label{eq:kder}
\[
\left( \partial_x^k - (-\partial_y)^k \right) \partial_z \G = 0 . \]
% \end{equation}
\end{conjecture}
We are able to prove \label{theo:rozn} under some additional assumptions,
however we believe it is true in general.

\section{Characterization of functions arising from bipartite graphs which are
polynomial}
\label{sec:proof}

\subsection{The main result}

\begin{theorem}
\label{theo:main}
Let $\G$ be a linear combination of bipartite graphs such that
\begin{equation}
\label{eq:main-assumption}
 \left( \partial_x^k - (-\partial_y)^k \right) \partial_z \G = 0
\end{equation}
for any integer $k > 0$.
Then
$\lambda\mapsto
N_\G^\lambda$ is a polynomial function on the set of
Young diagrams.

% \textbf{Wydaje mi się, że w tak sformułowanym twierdzeniu wynikanie mamy tylko
% w stronę $\Leftarrow$}
% 
% \textbf{To jest sluszna uwaga. Takie sformulowanie moze byc bledne (choc chyba
% nie jest, ale nie ma na to dowodu w tej pracy), aczkolwiek wynikanie
% $\Rightarrow$ jest w pewnym sensie prawdziwe (patrz wyciety kawalek o grafach
% komppletnych). Czy w takim razie zmieniamy tresc twierdzenia tak zeby miec cos
% na ksztalt iff, czy tylko wywalamy jedna implikacje? - M.}

% After some meditations I think that I really want to assume something
% different, namely
% for any $k\geq 0$
% \begin{equation}
% \label{eq:main-assumption2}
%  \left( \partial_x + \partial_y\right) \left( \partial_x + \partial_y\right)^k
% \partial_z G = 0 .
% \end{equation}
% 
% After more meditations I think that I want to assume, that for any $k\geq 0$
% \begin{equation}
% \label{eq:main-assumption2}
%  \left( \partial_x + \partial_y\right) \left( \partial_x -
% \partial_y\right)^k 
% \partial_z G = 0 .
% \end{equation}

\end{theorem}

The main idea of the proof is to find a connection between content-derivative
of a function $N_\G$ and a combinatorial derivation of the
underlying linear combination of bipartite graphs $\G$; we present it in the
following.

\subsection{Colorings of bipartite graphs with decorated edges}

Let a Young diagram $\lambda$ and a bipartite graph $G$ be given.
If an edge of $G$ is decorated by a real number $z$, we decorate
its white end by the number $\frac{\omega(z)+z}{2}$ (which is the $x$-coordinate
of the point at the profile of $\lambda$ with contents equal to $z$) and
we decorate its black end by the number $\frac{\omega(z)-z}{2}$ (which is the
$y$-coordinate
of the point at the profile of $\lambda$ with contents equal to $z$).
If some disjoint edges are decorated by $n$ real numbers $z_1, \dots, z_n$, then
% we requie that decorated edges create a matching and then 
we decorate white and black vertices in an analogous way.

For a bipartite graph $G$ with some disjoint edges decorated we define
$N_G(\lambda)$, the number of colorings of $\lambda$, as the volume of
the set of functions from undecorated vertices to $\R_+$ such that these
functions extended by values of decorated vertices are compatible with
$\lambda$.

% \begin{remark}
% For a bipartite graph $G$ with no decorated vertices $N_G(\lambda)$ is
% different
% than $N_G(\lambda)$, where $G$ has some decorated vertices. We use the same
% symbol for these two expressions, but it should be clear from the context if a
% graph $G$ has decorated vertices or not.
% \end{remark}

We will use the following lemma:
\begin{lemma}
\label{lem:der_z}
Let $(G,z)$ be a bipartite graph with one edge decorated by a real number $z$. Then
\begin{itemize}
% [label=(\roman*)]
 \item
\label{lem:der_z1}
$ \frac{d}{dz} N_{(G,z)}(\lambda) =
\frac{\omega'(z)+1}{2}N_{\partial_x(G,z)}(\lambda) +
\frac{\omega'(z)-1}{2}N_{\partial_y(G,z)}(\lambda)$, 
 \item
\label{lem:der_z2}
$\partial_{C_z} N_G = N_{\partial_z G}$.
\end{itemize}
\end{lemma}
The proof of this lemma is not difficult, but it is quite technical and we omit
it.

Using Lemma \ref{lem:der_z}, Theorem \ref{theo:rozn}, and results from Section
\ref{sec:diff} one can prove the following lemma: 

\begin{lemma}
Let the assumptions of Theorem \ref{theo:main} be fulfilled.
Then
\begin{itemize}
% [label=(\roman*)]
 \item
\label{lem:N_G-a}
$ z\mapsto  N_{\partial_z \G}(\lambda)  $ is a polynomial and
 \item
\label{lem:N_G-b}
$ \lambda \mapsto [z^i] N_{\partial_z \G}(\lambda) $
is a polynomial function on $\Y$ for any $i$.
\end{itemize}
\end{lemma}

\begin{proof}
The main ideas of the proof are the following. In order to show the first
property we are looking at $\frac{d^i}{dz^i} N_{\partial_z \G}$ and using Lemma
\ref{lem:der_z1} we can show that $\frac{d^i}{dz^i} N_{\partial_z \G} = 0$ for
any $i > |V| - 2$.
The proof of the second property is going by induction on
$i$ and it uses an Lemma \ref{lem:int} in a similar way like the proof of
Theorem \ref{theo:main} below. It is quite technical, so let us stop here.
\end{proof}

\subsection{Proof of the main result}

\begin{proof}[of Theorem \ref{theo:main}]
% If $N_\G$ is a polynomial function on $\Y$, then $N_\G$ is
% a
% polynomial in functionals $S_k$. Let $\G_k$ be a formal sum of bipartite
% graphs
% such that $S_k =
% N_{\G_k}$. It is enough to prove that $N_{(\partial_x + \partial_y)\partial_z
% \G_k} = 0$ for any
% integer $k > 1$. Proposition \ref{prop:d} finishes the proof
% of first implication.
We can assume without loss of
generality that every graph which contributes to
$\G$ has the same number of vertices, equal to $m$. Indeed, if this is not the
case,
we can write $\G=\G_2+\G_3+\cdots$ as a finite sum, where every graph
contributing to
$\G_i$ has $i$ vertices; then clearly \eqref{eq:main-assumption} is fulfilled
for every
$\G':=\G_i$.

\begin{figure}[tb]
\centerline{
    \begin{tikzpicture}[scale=0.6]
      \draw[->,thick] (-6,0) -- (6,0) node[anchor=west]{$z$};
      \draw[->,thick] (0,-0.4) -- (0,6) node[anchor=south]{$t$};
      \draw[line width=3pt] (-5.5,5.5) coordinate (left end) -- (-3.7,3.7)
coordinate (z1) -- ++(1.2,0.8) coordinate (z2) --
            ++(1.4,-1) coordinate (z3) -- ++(2,1) coordinate (z4) -- (4.1,4.1)
coordinate (z5) --
           (5.5,5.5) coordinate (right end);
      \draw[dotted] (left end) -- (0,0) -- (right end);
      \draw[dashed] (z1) -- (z1 |- 0,0) node[anchor=north]{$z_1$};
      \draw[dashed] (z2) -- (z2 |- 0,0) node[anchor=north]{$z_2$};
      \draw[dashed] (z3) -- (z3 |- 0,0) node[anchor=north]{$z_3$};
      \draw[dashed] (z4) -- (z4 |- 0,0) node[anchor=north]{$z_4$};
      \draw[dashed] (z5) -- (z5 |- 0,0) node[anchor=north]{$z_5$};
    \end{tikzpicture}
}
\caption{Piecewise-affine generalized Young diagram.}
\label{fig:affine}
\end{figure}
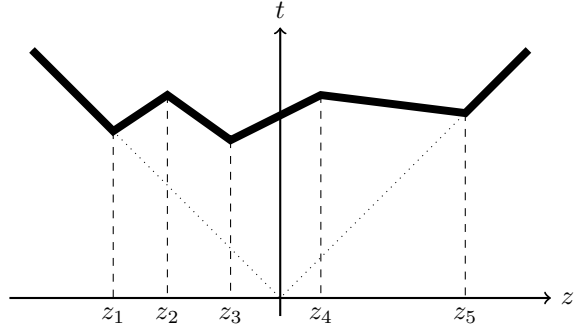

Assume that $\lambda$ is a piecewise affine generalized Young diagram such that
$|\omega'(z)|<1$ for any $z$ in the support of $\omega$ (see Figure
\ref{fig:affine}).
For any $t \in \R_+$ we define a generalized Young diagram $t \lambda$
which is a dilation of $\lambda$ by $t$. A profile $\tilde{\omega}$ of $t
\lambda$ is given by $\tilde{\omega}(s) = t\omega(s/t)$. By Lemma
\ref{lem:int} we can write:
\begin{multline*}
N_\G(\lambda) =  \frac{1}{m}\left. \frac{d}{dt}
N_\G^{t \lambda} \right|_{t=1}  = \frac{1}{2m} \int_\R 
 \frac{d}{dt}\big(t\omega(z/t)\big)\ \partial_{C_z}
N_\G(t\lambda)  \bigg|_{t=1} dz = \\
\frac{1}{2m} \int_\R \big(\omega(z)
-z\omega'(z)\big)\ \partial_{C_z}
N_\G(\lambda) dz.
\end{multline*}
Then, by Lemma \ref{lem:der_z} we have that
\begin{multline*}
N_\G(\lambda) = \frac{1}{2m} \int_\R \big(\omega(z) -z\omega'(z)\big) \sum_{0
\leq i \leq m-2} (i+1)z^i
\F_i(\lambda) dz = \\
\frac{1}{2m} \sum_{0 \leq i \leq m-2} \F_i(\lambda) \int_\R
\big(\omega(z) -z\omega'(z)\big) (i+1)z^i dz = 
\frac{1}{2m} \sum_{0 \leq i \leq m-2}
\F_i(\lambda)S_{i+2}(\lambda), 
\end{multline*}
where $\F_i = \frac{1}{i+1}[z^i]N_{\partial_z \G}(\lambda)$ is a polynomial
function on $\Y$
for each $i$ by Lemma \ref{lem:N_G-b}. It finishes the proof.
\end{proof}

\section{Applications}
\label{sec:appl}

\subsection{Bipartite maps}

A labeled (bipartite) graph drawn on a surface will be called a
\emph{(bipartite) map}. If this surface is orientable and its orientation is
fixed, then the underlying map is called \emph{oriented}; otherwise the map is
\emph{unoriented}. We will always assume that the surface is minimal in the
sense that after removing the graph from the surface, the latter becomes a
collection of disjoint open discs. If we draw an edge of such a graph with a fat
pen and then take its boundary, this edge splits into two \emph{edge-sides}. In
the above definition of the map, by \emph{`labeled'} we mean that each edge-side
is labeled with a number from the set $[2n]$ and each number from this set is
used exactly once. 

Each bipartite labeled map can be constructed by the following procedure. For a
partition $\lambda\vdash n$ we consider a family of $\ell(\lambda)$ bipartite
polygons with the number of edges given by partition $2\lambda = (2\lambda_1,
\dots, 2\lambda_{\ell(\lambda)})$. Then we label the edges of the polygons by
elements of $[2n]$ in such a way that each number is used exactly once. A
\emph{pair-partition} of $[2n]$ is defined as a family $P = \{V_1, \dots,
V_{n}\}$ of disjoint sets called \emph{blocks} of $P$, each containing exactly
two elements and such that $\bigcup P = [2n]$. For a given pair-partition $P$
we glue together each pair of edges of the polygons which is matched by $P$ in
such a way that a white vertex is glued with the other white one, and a black
vertex with the other black one.

\subsection{Normalized and zonal characters}

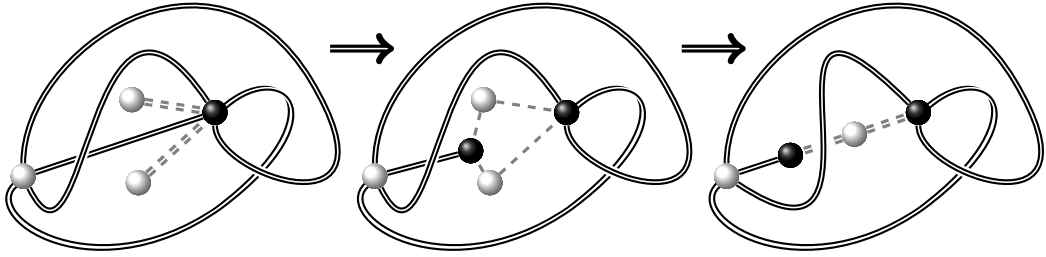
\begin{figure}
\begin{tikzpicture}[white2/.style={circle,ball color=white},
black1/.style={circle,draw=black,fill=black!80,thick,
inner sep=0pt,minimum size=4mm},
white1/.style={circle,draw=black,top color=white,bottom color=black!20, inner
sep=0pt,minimum size=4mm}, black2/.style={circle,ball
color=black}, auto, bend
right, scale = 0.85]

\begin{scope}
\draw[thick] (5,5) .. controls (6,6) and (7,5) .. (5,3.5) ..
controls (3,2.15)
and (1.2,3.55) .. (2,3.9);
\draw[white, line width=2pt] (5.1,5) .. controls (6,5.9) and (7.1,5) ..
(5.1,3.5) ..
controls (3,2)
and (1,3.5) .. (2,4);
\draw[thick] (5.1,5) .. controls (6,5.9) and (7.1,5) ..
(5.1,3.5) ..
controls (3,2)
and (1,3.5) .. (2,4);

\node[white2] (w1) at (2,4) {};
\node[white2] (w2) at (3.8,3.9) {};
\node[white2] (w3) at (3.7,5.2) {};
\node[black2] (b1) at (5,5) {};

\draw[double, thick] (w1) -- (b1);
\draw[gray, dashed, double, very thick] (w2) -- (b1);
\draw[gray, dashed, double, very thick] (w3) -- (b1);
\draw[white, line width=4pt] (w1) .. controls (3,2) and (3,8) .. (b1);
\draw[double, thick] (w1) .. controls (3,2) and (3,8) .. (b1);
\draw[white, line width=4pt] (b1) .. controls (5,4) and (8,3) ..
(6.5,5.5) ..
controls (5,8)
and (2,6) .. (w1);
\draw[double, thick] (b1) .. controls (5,4) and (8,3) ..
(6.5,5.5) ..
controls (5,8)
and (2,6) .. (w1);
\end{scope}

\begin{scope}[xshift = 5.5cm]
\draw[thick] (5,5) .. controls (6,6) and (7,5) .. (5,3.5) ..
controls (3,2.15)
and (1.2,3.55) .. (2,3.9);
\draw[white, line width=2pt] (5.1,5) .. controls (6,5.9) and (7.1,5) ..
(5.1,3.5) ..
controls (3,2)
and (1,3.5) .. (2,4);
\draw[thick] (5.1,5) .. controls (6,5.9) and (7.1,5) ..
(5.1,3.5) ..
controls (3,2)
and (1,3.5) .. (2,4);

\node[white2] (w1) at (2,4) {};
\node[white2] (w2) at (3.8,3.9) {};
\node[white2] (w3) at (3.7,5.2) {};
\node[black2] (b1) at (5,5) {};
\node[black2] (b2) at (3.5,4.4) {};

\draw[double, thick] (w1) -- (b2);
\draw[gray, dashed, very thick] (w2) -- (b1);
\draw[gray, dashed, very thick] (w3) -- (b1);
\draw[gray, dashed, very thick] (w2) -- (b2);
\draw[gray, dashed, very thick] (w3) -- (b2);
\draw[white, line width=4pt] (w1) .. controls (3,2) and (3,8) .. (b1);
\draw[double, thick] (w1) .. controls (3,2) and (3,8) .. (b1);
\draw[white, line width=4pt] (b1) .. controls (5,4) and (8,3) ..
(6.5,5.5) ..
controls (5,8)
and (2,6) .. (w1);
\draw[double, thick] (b1) .. controls (5,4) and (8,3) ..
(6.5,5.5) ..
controls (5,8)
and (2,6) .. (w1);
\end{scope}

\begin{scope}[xshift = 3cm]
\draw[thick] (13,5) .. controls (14,6) and (15,5) .. (13,3.5) ..
controls (11,2.15)
and (9.2,3.55) .. (10,3.9);
\draw[white, line width=2pt] (13.1,5) .. controls (14,5.9) and (15.1,5) ..
(13.1,3.5) ..
controls (11,2)
and (9,3.5) .. (10,4);
\draw[thick] (13.1,5) .. controls (14,5.9) and (15.1,5) ..
(13.1,3.5) ..
controls (11,2)
and (9,3.5) .. (10,4);

\node[white2] (w1') at (10,4) {};
\node[white2] (w2') at (12,4.66) {};
\node[black2] (b2') at (11,4.33) {};
\node[black2] (b1') at (13,5) {};

\draw[double, thick] (w1') -- (b2');
\draw[gray, dashed, double, very thick] (w2') -- (b1');
\draw[gray, dashed, double, very thick] (w2') -- (b2');
\draw[white, line width=4pt] (w1') .. controls (13,2) and (10,8) .. (b1');
\draw[double, thick] (w1') .. controls (13,2) and (10,8) .. (b1');
\draw[white, line width=4pt] (b1') .. controls (13,4) and (16,3) ..
(14.5,5.5) ..
controls (13,8)
and (10,6) .. (w1');
\draw[double, thick] (b1') .. controls (13,4) and (16,3) ..
(14.5,5.5) ..
controls (13,8)
and (10,6) .. (w1');
\end{scope}

\draw[->, very thick, double] (6.8,6) -- (7.8,6);

\begin{scope}[xshift=5.5cm]
\draw[->, very thick, double] (6.8,6) -- (7.8,6);
\end{scope}

\end{tikzpicture}
\caption{Example of a construction of a map and its subtree $(\tilde{\M},
\tilde{\T})$ (on the
right) from a given map with its subtree $(\M, \T)$ (on the left), such that
$\tilde{\M}/\tilde{\T} =
\M/\T$. Face type of maps is given by $\mu = (12)$. As pair-partitions we have
$\M = \{ \{1,7\}, \{2,3\}, \{4,6\}, \{5,11\}, \{8,9\}, \{10,12\}\}$, $\T =
\{\{2,3\}, \{8,9\}\}$, $\tilde{\M} = \{ \{1,7\}, \{2,8\}, \{3,9\}, \{4,6\},
\{5,11\}, \{10,12\}\}$, $\tilde{\T} = \{\{2,8\}, \{3,9\}\}$.}
\label{fig:maptree}
\end{figure}

\begin{theorem}[\cite{F'eray'Sniady2011}]
Let $\Sigma^{(\alpha)}_{\mu}$ denote the Jack character with parameter $\alpha$.
Then:
\begin{equation}
\label{alpha1}
 \Sigma^{(1)}_{\mu} = \sum_{\M}(-1)^{|V_b(\M)|}\ N_{\M},
\end{equation}
where the summation is
over all labeled bipartite oriented maps with the face type $\mu$ and
\begin{equation}
\label{alpha2}
 \Sigma^{(2)}_{\mu} = \sum_{\M}(-2)^{|V_b(\M)|}\ N_{\M},
\end{equation}
where the summation is over
all labeled bipartite maps (not necessarily oriented) with the face type $\mu$.
\end{theorem}
\begin{proof}
Due to the characterization of a Jack symmetric function which was given in
Section \ref{sub:characterization} it suffices to show that the right hand sides
of \eqref{alpha1} and \eqref{alpha2} satisfy conditions \ref{war-znikanie},
\ref{war-polynomial}, \ref{war-stopien}. Due to lack of space, instead of
\ref{war-znikanie} we will show a weaker condition \ref{war:i'weak}.

Definition of $N$ gives us property \ref{war-stopien} immediately.
Property \ref{war:i'weak} can be shown, as it was mentioned in
Section \ref{sub:characterization}, by proving that if we change functions $N$
on the right hand sides of \eqref{alpha1} and
\eqref{alpha2} by some other functions $\tilde{N}$ which count `injective
embeddings', the equalities will still hold. The proof of that will be the same
as in \cite{F'eray'Sniady2011}, hence we
omit it.

The novelty in the current proof is showing the property \ref{war-polynomial}.
First, we notice that 
$$\sum_{\M}(-2)^{|V_b(\M)|}N_{\M}(\lambda) =
\sum_{\M}(-1)^{|V_b(\M)|}N_{\M}(2\lambda).$$

In the following we shall prove that condition \eqref{eq:main-assumption} is
fulfilled.
Let us look at $\partial_x^k (\M, z)$ for some
bipartite map $\M$ with
one decorated edge by $z$. The procedure of derivation with respect to $x$ can
be viewed as taking all subtrees of $\M$ which consist of $k+1$ edges
connected by
a black vertex and where one edge is decorated by $z$ and collapsing them to one
decorated edge. Let us choose such a subtree $\T$. We can do the following
procedure with $\T$: we unglue every edges corresponding to $\T$ locally
in a
way that we create $k$ copies of a black vertex and local orientation of
each vertex is preserved; in this way we obtained locally a bipartite
$2k+2$-gon;
then we glue it again but in such a way that we glue white vertices together in
this $2k+2$-gon (see Figure \ref{fig:maptree}). 
% Because $\M$ is constructed by
% taking some pair-partition $P$ of the set of $2|\mu|$ elements, then $\T =
% \{\{e_{i_1}, e_{i_2}\},\{e_{i_3}, e_{i_4}\}\}$ as a subpartition of $P$. If we
% construct another pair-partition $\tilde{P}$ by changing blocks of $\T$ by
% $\tilde{\T} = \{\{e_{i_1}, e_{i_3}\},\{e_{i_2}, e_{i_4}\}\}$, then 
We obtained in this way a new bipartite map $\tilde{\M}$, such that
$|V_b(\tilde{\M})| =
|V_b(\M)| + k$ and which contains a subtree $\tilde{\T}$ with $k+1$ edges and
one
white vertex. Moreover, collapsing of $\T$ in
$\M$ to one decorated edge gives
us the same bipartite graph as collapsing of $\tilde{T}$ in $\tilde{\M}$ to
one decorated edge. We should check that this map has a face type $\mu$, but
this is clear from our construction. Of course we
can do the same procedure if we start from
$\partial_y^k (\M, z)$, because of the symmetry. These two procedures are
inverses of each other hence
\eqref{eq:main-assumption} holds true. Applying the Main Theorem
\ref{theo:main} to our case we obtain the property \ref{war-polynomial},
which finishes the proof.
\end{proof}

\acknowledgements
\label{sec:ack}

Research was supported by the
Polish Ministry of Higher Education research
grant N N201 364436 for the years 2009--2012.

\bibliographystyle{abbrvnat}
\bibliography{biblio2009}

\end{document}